\documentclass[reqno]{amsart}
\usepackage{amsmath}
\usepackage{amsthm, amscd, amssymb, amsfonts, amsbsy}
\usepackage[usenames, dvipsnames]{color}
\usepackage{enumerate}
\usepackage{hyperref}
\usepackage{verbatim}

\numberwithin{equation}{section}

\theoremstyle{plain}
\newtheorem{theorem}{Theorem}[section]
\newtheorem{lemma}[theorem]{Lemma}

\theoremstyle{definition}
\newtheorem{definition}[theorem]{Definition}
\newtheorem{assumption}[theorem]{Assumption}

\theoremstyle{remark}
\newtheorem{remark}[theorem]{Remark}

\makeatletter
\def\dashint{\operatorname%
{\,\,\text{\bf--}\kern-.98em\DOTSI\intop\ilimits@\!\!}}
\makeatother

\def\bR{\mathbb{R}}
\def\cL{\mathcal{L}}

\begin{document}
\title[Green function]{Estimates for Green functions of Stokes systems in two dimensional domains}

\author[J. Choi]{Jongkeun Choi}
\address[J. Choi]{School of Mathematics, Korea Institute for Advanced Study, 85 Hoegiro, Dongdaemun-gu, Seoul 02455, Republic of Korea}
\email{jkchoi@kias.re.kr}

\thanks{
J. Choi was supported by Basic Science Research Program 
through the National Research Foundation of Korea (NRF) funded by the Ministry of Education
(2017R1A6A3A03005168)}

\author[D. Kim]{Doyoon Kim}
\address[D. Kim]{Department of Mathematics, Korea University, 145 Anam-ro, Seongbuk-gu, Seoul, 02841, Republic of Korea}
\email{doyoon\_kim@korea.ac.kr}
\thanks{D. Kim was supported by Basic Science Research Program through the National Research Foundation of Korea (NRF) funded by the Ministry of Education (2016R1D1A1B03934369)}


\begin{abstract}
We prove the existence and pointwise bounds of the Green functions for stationary Stokes systems with measurable coefficients in two dimensional domains.
We also establish pointwise bounds of the derivatives of the Green functions under a regularity assumption on the $L_1$-mean oscillations of the coefficients.
\end{abstract}

\maketitle

\section{Introduction}		\label{S1}

Let $\Omega$ be a bounded domain in $\bR^2$ and $\cL$ be a differential operator in divergence form acting on column vector valued functions $u=(u^1, u^2)^{\top}$ as follows:
$$
\cL u=D_\alpha(A^{\alpha\beta}D_\beta u),
$$
where the coefficients $A^{\alpha\beta}$ are $2\times 2$ matrix-valued functions on $\Omega$, which satisfy the strong ellipticity condition \eqref{180409@eq2}.
The Green function of the operator $\cL$ is a pair $(G, \Pi)=(G(x,y), \Pi(x,y))$, where $G$ is a $2\times 2$ matrix-valued function and $\Pi$ is a $1 \times 2$ vector-valued function, such that if $(u, p)$ is a weak solution of the problem 
$$
\left\{
\begin{aligned}
\operatorname{div} u=g-(g)_\Omega &\quad \text{in }\, \Omega,\\
\cL^* u+\nabla p= f &\quad \text{in }\, \Omega,\\
u=0 &\quad \text{on }\, \partial \Omega
\end{aligned}
\right.
$$
with bounded data, then the solution $u$ is given by 
\begin{equation}		\label{180409@eq3}
u(y)=-\int_\Omega G(x,y)^{\top}f(x)\,dx+\int_\Omega \Pi(x,y)^\top g(x)\,dx.
\end{equation}
Here, $\cL^*$ is the adjoint operator of $\cL$ and $(g)_\Omega=\frac{1}{|\Omega|} \int_\Omega g\,dx$.
For a more precise definition of the Green function, see Section \ref{S2_3}.
We sometimes call this {\em{the Green function for the flow velocity of $\cL$}} 
because of the representation formula \eqref{180409@eq3} for the flow velocity $u$.

In this paper, we prove that if the divergence equation is solvable in a bounded domain $\Omega\subset \bR^2$ with an exterior measure condition \eqref{180405@eq4}, then there exists a unique Green function $(G, \Pi)$ of $\cL$ having the logarithmic pointwise bound
$$
|G(x,y)|\le C\bigg(1+\log \bigg(\frac{\operatorname{diam}(\Omega)}{|x-y|}\bigg)\bigg),  \quad \forall x,y\in \Omega, \quad x\neq y.
$$
For further details, see Theorem \ref{MT1}.
We emphasize that we do not impose any regularity assumptions on the coefficients $A^{\alpha\beta}$ of $\cL$.
Moreover, the assumption on the domain is sufficiently general to allow $\Omega$ to be, for example, a John domain with the exterior measure condition \eqref{180405@eq4}. 
Hence, the class of domains we consider includes Lipschitz domains, Reifenberg flat domains, and Semmes-Kenig-Toro (SKT) domains.
We also prove the following $L_\infty$-estimate away from $\partial \Omega$:
\begin{equation}		\label{180410@eq4}
\operatorname*{ess\, sup}_{B_{|x-y|/4}(x)} (|DG(\cdot,y)|+|\Pi(\cdot,y)|)\le C|x-y|^{-1}
\end{equation}
under the assumption that $A^{\alpha\beta}$ are of partially Dini mean oscillation (i.e., they are merely measurable in one direction and have Dini mean oscillations in the other direction). For further details, see Theorem \ref{MT2}.
The above estimate holds globally, i.e., 
\begin{equation}		\label{180410@eq4a}
|D_xG(x,y)|+|\Pi(x,y)|\le C|x-y|^{-1}, \quad \forall x,y\in \Omega, \quad x\neq y,
\end{equation}
when $A^{\alpha\beta}$ are of Dini mean oscillation in {\em{all}} the directions and $\Omega$ has a $C^{1,\rm{Dini}}$ boundary; see Theorem \ref{MT3}.
As far as the existence of the Green function is concerned, the coefficients $A^{\alpha\beta}$ need only be measurable.
Stokes systems with irregular coefficients of this type are partly motivated by the study of inhomogeneous fluids with density dependent viscosity and multiple fluids with interfacial boundaries; see \cite{MR0425391,MR1422251, MR2663713, MR3758532}.
Moreover, they can be employed to describe the motion of a laminar compressible viscous fluid; see \cite{arXiv:1805.00235}.

Green functions play an important role in the study of boundary value problems, in particular, in establishing the existence, uniqueness, and regularity of solutions to PDEs.
We refer the reader to \cite{MR0890159, MR3034453},  where the authors utilized Green function estimates for the existence and non-tangential maximal function estimates of harmonic functions satisfying certain boundary conditions.
In \cite{MR2353255, MR3086390}, the authors used the Green function for the uniqueness of solutions to elliptic equations.
Regarding the classical Stokes system, we refer to  \cite{MR0975121,MR1223521,MR2321139, MR2987056} for the usage of Green functions in establishing the existence of solutions with non-tangential or $L_p$-estimates.
By using our results in this paper, one may study the problems in the aforementioned papers for Stokes systems with variable coefficients in two dimensional irregular domains.

There is a large body of literature concerning Green functions of  Stokes systems.
With respect to the classical Stokes system 
$$
\Delta u+\nabla p=f,
$$
we refer the reader to  Ladyzhenskaya \cite{MR0254401}, Maz'ya-Plamenevski{\u\i} \cite{MR0725151, MR0734895}, Fabes-Kenig-Verchota \cite{MR0975121}, and D. Mitrea-I. Mitrea \cite{MR2763343}.
In \cite{MR0254401}, the author provided an explicit formula for the fundamental solution in two and three dimensions.
In \cite{MR0725151, MR0734895}, the authors established the existence and pointwise estimate of the Green function of a Dirichlet problem in a piecewise smooth domain in $\bR^3$.
The corresponding results were obtained in 
\cite{MR0975121} and \cite{MR2763343} on Lipschitz domains in $\bR^d$, where $d\ge 3$ and $d\ge 2$, respectively.
For Green functions of mixed problems, one can refer to the work of Maz'ya-Rossmann \cite{MR2182091} in three dimensional polyhedral domains and Ott-Kim-Brown \cite{MR3320459} in two dimensional Lipschitz domains.
Regarding Stokes systems with variable coefficients
$$
\cL u+\nabla p=f,
$$
we refer the reader to \cite{MR3693868,MR3670039, arXiv:1804.10588}.
In \cite{MR3693868}, the authors established the existence and  pointwise estimate of the Green function of a Dirichlet problem in a bounded $C^1$ domain when $d\ge 3$ and the coefficients of $\cL$ have vanishing mean oscillations.
The corresponding results were obtained in \cite{MR3670039} on the whole space and a half space when coefficients are merely measurable in one direction and have small mean oscillations in the other directions (partially BMO).
In \cite{arXiv:1804.10588}, the authors constructed Green function for a conormal derivative problem when coefficients are variably partially BMO.
We also refer the reader to \cite{arXiv:1710.05383} for Green functions of Stokes systems with oscillating periodic coefficients.

Note that all of the above mentioned results for Green functions of Stokes systems with variable coefficients are limited to the case that $d\ge 3$.
In this paper, as mentioned as an interesting problem in \cite{MR3320459}, we extend and apply the method used in the construction of Green function of the classical Stokes system to Stokes systems with non-constant coefficients when $d=2$.
Because we are unable to find any literature dealing with Green functions of Stokes systems with {\em variable coefficients} in two dimensional domains, we anticipate that our results fill a gap in the literature for the two dimensional case.
The only literature we have found is a recent paper \cite{CD3}, where the authors treated a Green function for the representation formula of the pressure when $d\ge 2$ and coefficients are of (partially) Dini mean oscillation.
Indeed, the method employed in this paper is applicable to higher dimensional cases, but it is questionable whether one can obtain the same generality achieved in \cite{MR3693868,MR3670039, arXiv:1804.10588} by using the approach in this paper.
In particular, to establish the existence of Green functions for $d \geq 3$ following the steps in this paper, one needs global $W_q^1$-estimates, $q > d$, which require stronger regularity assumptions on the coefficients and on the boundary of the domain than those, for instance, in \cite{arXiv:1804.10588}.
On the other hand, for the two dimensional case with variable coefficients, one may consider  applying the method used in \cite{MR3693868,MR3670039, arXiv:1804.10588}, which is based on local $C^\alpha$ and $L_\infty$-estimates, as well as a global $W^1_2$-estimate.
Here, by the global $W^{1}_2$-estimate we mean
\begin{equation}	\label{181017@eq1}
\|Du\|_{L_2(\Omega)}+\|p\|_{L_2(\Omega)}\le C\big(\|f\|_{L^{2^\#}(\Omega)}+\|f_\alpha\|_{L^2(\Omega)}+\|g\|_{L^2(\Omega)}\big),
\end{equation}   
where $d\ge 3$, $2^\#=\frac{2d}{d+2}$,
and $(u, p)$ is a weak solution of
$$
\left\{
\begin{aligned}
\cL u+\nabla p=f+D_\alpha f_\alpha &\quad \text{in }\, \Omega,\\
\operatorname{div}u=g &\quad \text{in }\, \Omega,
\end{aligned}
\right.
$$
with some boundary condition.
The estimate \eqref{181017@eq1} is optimal in the sense that the constant $C$ does not depend on the size of $\Omega$ when $\Omega$ is a ball. 
However, if $d=2$, such an estimate is not true.
Indeed, the estimate \eqref{181017@eq1} holds with $q>1$ in place of $2^\#=1$, which is not optimal in the  aforementioned sense, nor is well suited to providing necessary estimates of the Green function when $d=2$.
Thus, to apply the method used in the higher dimensional case to the two dimensional case, we need some modifications, which seem inevitable because the higher dimensional Green function and the two dimensional Green function have different types of pointwise bounds.
Rather than modifying the method for $d \geq 3$, in this paper we take a straightforward approach so that
 we directly derive the Green function. 
 We recall that in the higher dimensional case, the Green functions are obtained by an approximation argument.

Some remarks are in order regarding the approach in this paper. In fact, there are several paths to constructing Green functions with logarithmic growth for Stokes systems and elliptic systems in two dimensional domains.
In many references, for instance \cite{MR1354111, MR3320459},
the construction of Green functions relies on the existence of a solution with gradient estimates in the weak Lebesgue space $L_{2,\infty}(\Omega)$.
In this paper, we derive the gradient estimates by adapting the idea of Dolzmann-M\"uller \cite{MR1354111}, where the authors constructed Green functions for elliptic systems with the zero Dirichlet boundary condition on bounded domains in $\bR^d$ ($d\ge 2$) with a $C^1$ or Lipschitz boundary.
For the $L_{2,\infty}$-estimate, we utilize $W_q^1$-estimate and solvability for the Stokes system together with real interpolation, where the $W^1_q$-estimate follows from the reverse H\"older's inequality.
In \cite{MR3320459}, by using complex interpolation the authors derived the $L_{2,\infty}$-estimate for the Green function of the classical Stokes system with a mixed boundary condition in a Lipschitz domain.
For another approach to constructing Green functions, we refer the reader to \cite{MR2485428}, where the authors construct Green functions for elliptic systems in a (possibly unbounded) domain in $\bR^2$ by integrating parabolic Green functions in $t$ variable.

The estimates of Green functions are closely related to the regularity theory of solutions.
In particular, for the bounds of the derivatives of the Green function such as \eqref{180410@eq4} and \eqref{180410@eq4a}, solutions of the system are required to have bounded gradients, which are not available for Stokes systems and elliptic systems with measurable coefficients.
For this reason, we need to impose certain regularity assumptions on the coefficients and domains.
In this paper, for the estimates \eqref{180410@eq4} and \eqref{180410@eq4a}, we utilize the results given in \cite{arXiv:1803.05560, arXiv:1805.02506}, where the authors proved $W_\infty^1$ and $C^1$-estimates for Stokes systems with coefficients having (partially) Dini mean oscillations.

The remainder of this paper is organized as follows.
We introduce some notation and definitions in the next section.
In Section \ref{S3}, we state the main theorems.
In Section \ref{S4}, we present some auxiliary results, and in Section \ref{S5}, we provide the proofs of the main theorems.

\section{Preliminaries}		\label{S2}

\subsection{Notation}		

Throughout the paper we denote by $\Omega$ a bounded domain in the Euclidean space $\bR^2$.
For any $x\in \bR^2$ and $r>0$, we write $\Omega_r(x)=\Omega \cap B_r(x)$, where $B_r(x)$ is the usual Euclidean disk of radius $r$ centered at $x$.
For $q\in [1,\infty]$, we denote by $W^1_q(\Omega)$ the usual Sobolev space and $\mathring{W}^{1}_q(\Omega)$ the closure of $C^\infty_0(\Omega)$ in $W^1_q(\Omega)$.
We also define the weak $L_q$ space, denoted by $L_{q,\infty}(\Omega)$, as the set of 
all measurable functions on $\Omega$ having a finite quasi-norm
$$
\|u\|_{L_{q,\infty}(\Omega)}=\sup_{t>0} t\big|\{x\in \Omega:|u(x)|>t\}\big|^{1/q}.
$$
We define 
$$
\tilde{L}_q(\Omega)=\{u\in L_q(\Omega):(u)_\Omega=0\},
$$
$$
\tilde{L}_{q,\infty}(\Omega)=\{u\in L_{q,\infty}(\Omega):(u)_\Omega=0\},
$$
$$
\tilde{W}^1_q(\Omega)=\{u\in W^1_q(\Omega):(u)_\Omega=0\},
$$
where $(u)_\Omega$ is the average of $u$ over $\Omega$, i.e.,
$$
(u)_\Omega=\dashint_\Omega u\,dx=\frac{1}{|\Omega|}\int_\Omega u\,dx.
$$
We recall that 
$$
\|u+v\|_{L_{q,\infty}(\Omega)}\le 2\big(\|u\|_{L_{q,\infty}(\Omega)}+\|v\|_{L_{q,\infty}(\Omega)}\big) \quad \text{for all }\,  u,v\in L_{q,\infty}(\Omega)
$$
and
$$
L_{q,\infty}(\Omega)\subset L_s(\Omega) \subset L_{s,\infty}(\Omega) \quad \text{for } \, s<q.
$$

We say that a measurable function $\omega:(0,a]\to [0, \infty)$ is a Dini function provided that there are constants $c_1,\, c_2>0$ such that 
$$
c_1\omega(t)\le \omega(s)\le c_2\omega(t) \quad \text{whenever }\, 0<\frac{t}{2}\le s\le t\le a
$$
and that $\omega$ satisfies the Dini condition
$$
\int_0^a \frac{\omega(t)}{t}\,dt<\infty.
$$

\begin{definition}		\label{D2}
Let $f\in L^1(\Omega)$.
\begin{enumerate}[$(a)$]
\item
We say that $f$ is of {\em{partially Dini mean oscillation}} in (the interior of) $\Omega$ if there exists a Dini function $\omega:(0,1]\to [0,\infty)$ such that for any $x=(x_1,x_2)\in \Omega$ and $r\in (0,1]$ satisfying $B_{2r}(x)\subset \Omega$, we have 
$$
\dashint_{B_r(x)}\bigg|f(y)-\dashint_{B_r'(x_2)}f(y_1,s)\,ds\bigg|\,dy\le \omega(r),
$$
where $B_r'(x_2)=\{t\in \bR:|t-x_2|<r\}$.
\item
We say that $f$ is of {\em{Dini mean oscillation}} in $\Omega$ if there exists a Dini function $\omega:(0, 1]\to [0,\infty)$ such that for any $x\in \overline{\Omega}$ and $r\in (0,1]$, we have 
$$
\dashint_{\Omega_r(x)}\bigg|f(y)-\dashint_{\Omega_r(x)} f(z)\,dz \bigg|\,dy\le \omega(r).
$$
\end{enumerate}
\end{definition}

We define a $C^{1,\rm{Dini}}$ domain by locally the graph of a $C^1$ function whose derivatives are uniformly Dini continuous.

\begin{definition}		\label{D3}
We say that $\Omega$ has a $C^{1,\rm{Dini}}$ boundary if there exist a constant $R_0\in (0,1]$ and a Dini function $\varrho_0:(0, 1]\to [0,\infty)$ such that the following holds:
For any $z=(z_1,z_2)\in \partial \Omega$, there exist a $C^1$ function $\chi:\bR \to \bR$ and a coordinate system depending on $z$, such that in the new coordinate system we have 
$$
|\chi'(z_2)|=0, \quad \Omega_{R_0}(z)=\{x=(x_1,x_2)\in B_{R_0}(z): x_1>\chi(x_2)\},
$$
and 
$$
\varrho_{\chi}(r)\le \varrho_0(r) \quad \text{for all }\, r\in (0,R_0),
$$
where $\varrho_{\chi}$ is the modulus of continuity of $\chi'$, i.e.,
$$
\varrho_{\chi}(r)=\sup \big\{| \chi'(s)- \chi'(t)| : s,t\in \bR, \, |s-t|\le r \big\}.
$$
\end{definition}

\subsection{Stokes system}	

Let $\cL$ be a strongly elliptic operator of the form
$$
\cL u=D_\alpha(A^{\alpha\beta}D_\beta u),
$$
where the coefficients $ A^{\alpha\beta}= A^{\alpha\beta}(x)$ are $2\times 2$ matrix-valued functions on $\Omega$ with entries $A_{ij}^{\alpha\beta}$ satisfying the strong ellipticity condition, i.e., there is a constant $\lambda\in (0,1]$ such that 
\begin{equation}		\label{180409@eq2}
|A^{\alpha\beta}(x)|\le \lambda^{-1}, \quad \sum_{\alpha,\beta=1}^2A^{\alpha\beta}(x)\xi_\beta\cdot \xi_\alpha\ge \lambda\sum_{\alpha=1}^2|\xi_\alpha|^2
\end{equation}
for any $x\in \Omega$ and $\xi_\alpha\in \bR^2$, $\alpha\in \{1,2\}$.
We do not assume that the coefficients $A^{\alpha\beta}$ are symmetric.
The adjoint operator $\cL^*$ is given by 
$$
\cL^* u=D_\alpha((A^{\beta\alpha})^{\top} D_\beta u),
$$
where $(A^{\beta\alpha})^{\top}$ is the transpose of the matrix $A^{\beta\alpha}$ for each $\alpha,\beta\in \{1,2\}$.

Let $f\in L_{q_1}(\Omega)^2$ and $f_\alpha \in L_q(\Omega)^2$, where 
$q,q_1\in (1,\infty)$ and $q_1\ge 2q/(q+2)$.
We say that $(u,p)\in \mathring{W}^1_q(\Omega)^2\times L_q(\Omega)$ is a weak solution of the problem
$$
\cL u+\nabla p=f+D_\alpha f_\alpha \quad  \text{in }\, \Omega,
$$
if 
$$
\int_\Omega A^{\alpha\beta}D_\beta u\cdot D_\alpha \phi \,dx+\int_\Omega p\operatorname{div} \phi\,dx=-\int_\Omega f\cdot \phi\,dx+\int_\Omega f_\alpha\cdot D_\alpha \phi\,dx
$$
holds for any $\phi\in \mathring{W}^1_{q/(q-1)}(\Omega)^2$.
Similarly, we say that $(u, p)\in \mathring{W}^1_q(\Omega)^2\times L_q(\Omega)$ is a weak solution of 
$$
\cL^* u+\nabla p=f+D_\alpha f_\alpha \quad \text{in }\, \Omega,
$$
if 
$$
\int_\Omega A^{\alpha\beta}D_\beta \phi\cdot D_\alpha u \,dx+\int_\Omega p\operatorname{div} \phi\,dx=-\int_\Omega f\cdot \phi\,dx+\int_\Omega f_\alpha\cdot D_\alpha \phi\,dx
$$
holds for any $\phi\in \mathring{W}^1_{q/(q-1)}(\Omega)^2$.

\subsection{Green function for the flow velocity}	\label{S2_3}

The following is the definition of  the Green function of Stokes system.
Here, $G=G(x,y)$ is a $2\times 2$ matrix-valued function and $\Pi=\Pi(x,y)$ is a $1 \times 2$ vector-valued function. 

\begin{definition}		\label{D1}
We say that a pair $(G,\Pi)$ is the Green function (for the flow velocity) of  $\cL$ in $\Omega$ if it satisfies the following properties.
\begin{enumerate}[$(i)$]
\item
For any $y\in \Omega$ and $r>0$, 
$$
G(\cdot,y)\in \mathring{W}^1_1(\Omega)^{2\times 2}, \quad 
\Pi(\cdot,y)\in  \tilde{L}_1(\Omega)^2.
$$
\item
For any $y\in \Omega$, $(G(\cdot,y), \Pi(\cdot,y))$ satisfies 
$$
\left\{
\begin{aligned}
\operatorname{div} G(\cdot,y)=0 &\quad \text{in }\, \Omega,\\
\cL G(\cdot,y)+\nabla \Pi(\cdot,y)=-\delta_yI &\quad \text{in }\, \Omega,
\end{aligned}
\right.
$$
in the sense that for $k\in \{1,2\}$ and $\phi\in \mathring{W}^1_\infty(\Omega)^2\cap C(\Omega)^2$, we have 
$$
\operatorname{div} G^{\cdot k}(\cdot,y)=0 \quad \text{in }\, \Omega
$$
and 
$$
\int_\Omega A^{\alpha\beta} D_\beta G^{\cdot k}(\cdot ,y)\cdot D_\alpha \phi\,dx+\int_\Omega \Pi^k(\cdot ,y)\operatorname{div}\phi \,dx=\phi^k(y),
$$
where $G^{\cdot k}(\cdot,y)$ is the $k$-th column of $G(\cdot,y)$.
\item
If $(u,p)\in \mathring{W}^1_2(\Omega)^2\times L_2(\Omega)$ is a weak solution of the problem 
\begin{equation}		\label{171010@eq1}
\left\{
\begin{aligned}
\operatorname{div} u=g &\quad \text{in }\, \Omega,\\
\cL^* u+\nabla p=f+D_\alpha f_\alpha &\quad  \text{in }\, \Omega,
\end{aligned}
\right.
\end{equation}
where  $f,f_\alpha\in L_\infty(\Omega)^2$ and $g\in \tilde{L}_\infty(\Omega)$, 
then for a.e. $y\in \Omega$, we have 
\begin{equation}		\label{170519@eq3}
\begin{aligned}
u(y)&=-\int_\Omega G(x,y)^{\top}f(x)\,dx\\
&\quad +\int_\Omega D_\alpha G(x,y)^{\top} f_\alpha(x)\,dx+\int_\Omega \Pi(x,y)^\top g(x)\,dx,
\end{aligned}
\end{equation}
where $G(x,y)^{\top}$ and $\Pi(x,y)^{\top}$ are the transposes of  $G(x,y)$ and $\Pi(x,y)$.
\end{enumerate}
The Green function of the adjoint operator $\cL^*$ is defined similarly.
\end{definition}

\begin{remark}
The $W^1_2$-solvability of Stokes system (Lemma \ref{170517@lem1}) and the property $(iii)$ in the above definition ensure the uniqueness of the Green function in the following sense.
 Let $(\tilde{G}, \tilde{\Pi})$ be another Green function satisfying the properties in Definition \ref{D1}. 
Then for each $\phi\in C^\infty_0(\Omega)^2$ and $\varphi\in C^\infty_0(\Omega)$, there exists a measure zero set $N\subset \Omega$ such that for any $y\in \Omega\setminus N$, we have 
$$
\int_\Omega \big(G(x,y)-\tilde{G}(x,y)\big)\phi(x)\,dx=\int_\Omega \big(\Pi(x,y)-\tilde{\Pi}(x,y)\big)\varphi(x)\,dx=0.
$$
\end{remark}

\section{Main results}			\label{S3}

The main results of this paper are as follows.
To establish the existence and pointwise bound of the Green function of Stokes system, we impose the following assumption stating that the divergence equation is solvable, which is valid on, for instance, a John domain; see \cite[Theorem 4.1]{MR2263708}.
As mentioned in Section \ref{S2}, $\Omega$ is a bounded domain in $\bR^2$.

\begin{assumption}		\label{A1}
There exists a constant $K_0>0$ such that the following holds:
For any $g\in \tilde{L}_2(\Omega)$, there exists $u\in \mathring{W}^1_2(\Omega)^2$ such that 
$$
\operatorname{div} u=g \quad \text{in }\, \Omega,\quad \|Du\|_{L_2(\Omega)}\le K_0\|g\|_{L_2(\Omega)}.
$$
\end{assumption}

\begin{theorem}		\label{MT1}
Let $\Omega$ satisfy
\begin{equation}		\label{180405@eq4}
|B_R(x_0)\setminus \Omega|\ge \theta R^2, \quad \forall x_0\in \partial \Omega, \quad \forall R\in (0, 1].
\end{equation}
Then under Assumption \ref{A1}, there exist Green functions $(G,\Pi)$ of $\cL$ and $(G^*, \Pi^*)$ of $\cL^*$.
Moreover, $G$ and $G^*$ are continuous in $\{(x,y)\in \Omega\times \Omega: x\neq y\}$ and satisfy 
\begin{equation}		\label{170526@eq1}
G(x,y)=G^*(y,x)^{\top} \quad \text{for all }\, x,y\in \Omega, \quad x\neq y.
\end{equation}
Furthermore, we have the following estimates.
\begin{enumerate}[$(a)$]
\item
For any $y\in \Omega$, we have
\begin{equation}		\label{170807@eq6}
\|DG(\, \cdot \, , y)\|_{L_{2,\infty}(\Omega)}+ \|\Pi(\, \cdot \, , y)\|_{L_{2,\infty}(\Omega)}\le C.
\end{equation}
\item
There exists $q_0=q_0(\lambda,\theta, K_0)>2$ such that for any $x\in \bR^2$, $y\in \Omega$, $0<R<\operatorname{diam}(\Omega)$ satisfying $|x-y|>R$,
we have 
\begin{equation}		\label{180404@A2}
\|DG(\cdot,y)\|_{L_{q_0}(\Omega_{R/2}(x))}+\|\Pi(\cdot,y)\|_{L_{q_0}(\Omega_{R/2}(x))}\le CR^{-\gamma}
\end{equation}
and 
\begin{equation}		\label{180404@A3}
\left[G(\cdot,y)\right]_{C^{\gamma}(\Omega_{R/2}(x))}\le CR^{-\gamma},
\end{equation}
where $\gamma=1-2/q_0$.
\item
For any $x,y\in \Omega$ with $x\neq y$,  we have
\begin{equation}		\label{170807@eq7}
|G(x,y)|\le C\bigg(1+\log\bigg(\frac{\operatorname{diam}(\Omega)}{|x-y|}\bigg)\bigg).
\end{equation}
\end{enumerate}
In the above, the constants $C$ depend only on $\lambda$, $\theta$, $K_0$, and $\operatorname{diam}(\Omega)$.
\end{theorem}

\begin{remark}		\label{180409@rmk1}
In Theorem \ref{MT1}, because $G(\cdot,y)$ satisfies the zero Dirichlet boundary condition, we have a better estimate than \eqref{170807@eq7} near the boundary of $\Omega$.
Indeed, by \eqref{180404@A3} and $G(\cdot,y)\equiv 0$ on $\partial \Omega$, it is easily seen that for any $x,y\in \Omega$ with $\operatorname{dist}(x, \partial \Omega)\le  |x-y|/4$, we have 
$$
|G(x,y)|\le C \frac{{\operatorname{dist}(x,\partial \Omega)}^\gamma}{|x-y|^\gamma},
$$
where $C=C(\lambda, \theta, K_0, \operatorname{diam}(\Omega))$.

\end{remark}

\begin{remark}		\label{RMK2}
Let $(u,p)\in \mathring{W}^1_2(\Omega)^2\times L_2(\Omega)$ be a weak solution of the problem 
$$
\left\{
\begin{aligned}
\operatorname{div} u=0 &\quad \text{in }\, \Omega,\\
\cL u+\nabla p=f+D_\alpha f_\alpha &\quad  \text{in }\, \Omega,
\end{aligned}
\right.
$$
where $f,f_\alpha\in L_\infty(\Omega)^2$.
Then by using \eqref{170526@eq1} and the counterpart of $(iii)$ in Definition \ref{D1} for $(G^*, \Pi^*)$, we have 
$$
\begin{aligned}
u(y)&=-\int_\Omega G(y,x)f(x)\,dx+\int_\Omega D_\alpha G(y,x) f_\alpha(y)\,dy
\end{aligned}
$$
for a.e. $y\in \Omega$.
\end{remark}

In the theorem below, we prove an interior $L^\infty$-estimate for $(DG, \Pi)$ when the coefficients of $\cL$ are of partially Dini mean oscillation.

\begin{theorem}		\label{MT2}
Let $\Omega$ satisfy \eqref{180405@eq4} and $(G, \Pi)$ be the Green function of $\cL$ in $\Omega$ constructed in Theorem \ref{MT1} under Assumption \ref{A1}.
Suppose that the coefficients $A^{\alpha\beta}$ of $\cL$ are of partially Dini mean oscillation in $\Omega$ satisfying Definition \ref{D2} $(a)$ with a Dini function $\omega=\omega_A$.
Then for any $x,y\in \Omega$ with $0<|x-y|\le \frac{1}{2}\operatorname{dist}(y, \partial \Omega)$, we have 
\begin{equation}		\label{180405@C2}
\operatorname*{ess\, sup}_{B_{|x-y|/4}(x)} (|DG(\cdot,y)|+|\Pi(\cdot,y)|)\le C|x-y|^{-1},
\end{equation}
where $C=C(\lambda,\theta,K_0,\operatorname{diam}(\Omega),\omega_A)$.
\end{theorem}

\begin{remark}
							\label{rem0517_1}
In Theorem \ref{MT2}, if we assume further that $A^{\alpha\beta}$ are of Dini mean oscillation with respect to {\em{all}} the directions in $\Omega$ satisfying Definition \ref{D2} $(b)$ as in Theorem \ref{MT3} below (but without the $C^{1,\operatorname{Dini}}$ regularity assumption on the boundary in Theorem \ref{MT3}), we obtain an estimate as in \eqref{180405@F1} below, but only in the interior of the domain. 
Indeed, by Definition \ref{D1} $(ii)$ and \eqref{180405@C1}, we see that $DG(\cdot,y)$ and $\Pi(\cdot,y)$ are continuous in $\Omega\setminus \{y\}$.
Hence, ``$\operatorname*{ess\, sup}$'' in \eqref{180405@C2} can be replaced by ``sup''.
Therefore, for any $x,y\in \Omega$ with $0<|x-y|\le \frac{1}{2}\operatorname{dist}(y,\partial \Omega)$, we have 
$$
|D_xG(x,y)|+|\Pi(x,y)|\le C|x-y|^{-1},
$$
where $C=C(\lambda, \theta, K_0, \operatorname{diam}(\Omega), \omega_A)$.
\end{remark}

In the next theorem, we prove a global pointwise bound for $(DG, \Pi)$ when the coefficients of $\cL$ are of Dini mean oscillation and $\Omega$ has a $C^{1,\rm{Dini}}$ boundary.

\begin{theorem}		\label{MT3}
Let $\Omega$ have a $C^{1,\rm{Dini}}$ boundary as in Definition \ref{D3}.
Let $(G, \Pi)$ be the Green function of $\cL$ in $\Omega$ constructed in Theorem \ref{MT1}.
Suppose that the coefficients $A^{\alpha\beta}$ of $\cL$ are of Dini mean oscillation in $\Omega$ satisfying Definition \ref{D2} $(b)$ with a Dini function $\omega=\omega_A$.
Then for any $y\in \Omega$ and $R>0$, we have 
\begin{equation}
							\label{eq0517_01}
(G(\cdot,y), \Pi(\cdot,y))\in C^1(\overline{\Omega\setminus B_R(y)})^{2\times 2}\times C(\overline{\Omega\setminus B_R(y)})^2.
\end{equation}
Moreover, for any $x,y\in \Omega$ with $x\neq y$, we have 
\begin{equation}		\label{180405@F1}
|D_xG(x,y)|+|\Pi(x,y)|\le C|x-y|^{-1},
\end{equation}
where $C=C(\lambda, \operatorname{diam}(\Omega),\omega_A, R_0, \varrho_0)$.
\end{theorem}

\begin{remark}
In the above theorem, the existence of the Green function follows from Theorem \ref{MT1} because $\Omega$ satisfies Assumption \ref{A1} and \eqref{180405@eq4}.
Indeed, by \cite[Theorem 4.1]{MR2263708}, $\Omega$ satisfies Assumption \ref{A1} with $K_0=K_0(\operatorname{diam}(\Omega), R_0, \varrho_0)$ because $\Omega$ is a John domain as in \cite[Definition 2.1]{MR2263708} with respect to $(x_0, L)=(x_0, L)(R_0, \varrho_0)$.
Moreover, $\Omega$ satisfies \eqref{180405@eq4} with $\theta=\theta(R_0, \varrho_0)$ owing to the properties in Definition \ref{D3}.
\end{remark}

\section{Auxiliary results}		\label{S4}

In this section, we prove some auxiliary results.
We do not impose any regularity assumptions on the coefficients $A^{\alpha\beta}$ of the operator $\cL$.
The following lemma concerns the solvability of Stokes system in $\mathring{W}^1_2(\Omega)^2\times \tilde{L}_2(\Omega)$.

\begin{lemma}		\label{170517@lem1}
Let $f_\alpha\in L_2(\Omega)^2$ and $g\in \tilde{L}_2(\Omega)$.
Then under Assumption \ref{A1}, there exists a unique $(u,p)\in \mathring{W}^1_2(\Omega)^2\times \tilde{L}_2(\Omega)$ satisfying 
\begin{equation}		\label{170517_eq1}
\left\{
\begin{aligned}
\operatorname{div} u=g &\quad \text{in }\, \Omega,\\
\cL u+\nabla p=D_\alpha f_\alpha &\quad \text{in }\, \Omega.
\end{aligned}
\right.
\end{equation}
Moreover, we have
\begin{equation}		\label{180403@eq1}
\|Du\|_{L_2(\Omega)}+\|p\|_{L_2(\Omega)}\le C\big(\|f_\alpha\|_{L_2(\Omega)}+\|g\|_{L_2(\Omega)}\big),
\end{equation}
where $C=C(\lambda, K_0)$.
\end{lemma}

\begin{proof}
See, for instance, \cite[Lemma 3.2]{MR3693868}, where the authors proved the solvability of the Stokes system \eqref{170517_eq1} with $f+D_\alpha f_\alpha$ in place of $D_\alpha f_\alpha$, and the $L_2$-estimate
$$
\|Du\|_{L_2(\Omega)}+\|p\|_{L_2(\Omega)}\le C\big(\|f\|_{L_2(\Omega)}+\|f_\alpha\|_{L_2(\Omega)}+\|g\|_{L_2(\Omega)}\big),
$$
where $C=C(\lambda, K_0, |\Omega|)$.
From the proof of \cite[Lemma 3.2]{MR3693868}, it is easily seen that if $f\equiv 0$, then the constant $C$ depends only on $\lambda$ and $K_0$.
We omit the details.
\end{proof}

\begin{lemma}		\label{170807@lem1}
Let $\Omega$ satisfy
\begin{equation}		\label{180403@eq2}
|B_R(x_0)\setminus \Omega|\ge \theta R^2, \quad \forall x_0\in \partial \Omega, \quad \forall R\in (0,1].
\end{equation}
Let $f_\alpha\in L_{\infty}(\Omega)^2$, $g\in \tilde{L}_{\infty}(\Omega)$, and  $(u,p)\in \mathring{W}^1_2(\Omega)^2\times \tilde{L}_2(\Omega)$ be the weak solution of \eqref{170517_eq1} derived from Lemma \ref{170517@lem1} under Assumption \ref{A1}.
Then for $x_0\in \overline{\Omega}$ and $R\in (0, 1]$  satisfying either
$$
B_R(x_0)\subset \Omega \quad \text{or}\quad x_0\in \partial \Omega,
$$
we have 
\begin{align*}
\big(|D\bar{u}|^2+|\bar{p}|^2\big)^{1/2}_{B_{R/2}(x_0)}&\le C\big(|D\bar{u}|^{q_0}+|\bar{p}|^{q_0}\big)^{1/q_0}_{B_{R}(x_0)}+C \big(|\bar{f}_\alpha|^{2}+|\bar{g}|^{2}\big)^{1/2}_{B_{R}(x_0)},
\end{align*}
where $q_0\in (1,2)$ and $C=C(\lambda,\theta,K_0, q_0)$.
Here, $\bar{u}$, $\bar{p}$, $\bar{f}_\alpha$, and $\bar{g}$ are the extensions of $u$, $p$, $f_\alpha$, and $g$ to $\bR^2$ so that they are zero on $\bR^2\setminus \Omega$.
\end{lemma}

\begin{proof}
For the proof of the lemma, we refer the reader to that of \cite[Lemma 3.5]{MR3758532}, where the authors proved the same inequality for the Stokes system with measurable coefficients in a Reifenberg flat domain. 
We note that  Reifenberg flat domains satisfy \eqref{180403@eq2} and Assumption \ref{A1}.
The argument in the  proof of \cite[Lemma 3.5]{MR3758532} is sufficiently general to allow the  domain $\Omega$ to satisfy \eqref{180403@eq2} and Assumption \ref{A1}. 
\end{proof}

We obtain the following reverse H\"older's inequality.

\begin{lemma}		\label{170807@lem2}
Let $\Omega$ satisfy \eqref{180403@eq2}.
Let $f_\alpha\in L_{\infty}(\Omega)^2$, $g\in \tilde{L}_{\infty}(\Omega)$, and  $(u,p)\in \mathring{W}^1_2(\Omega)^2\times \tilde{L}_2(\Omega)$ be the weak solution of \eqref{170517_eq1} derived from Lemma \ref{170517@lem1} under Assumption \ref{A1}.
Then there exists $\varepsilon_0=\varepsilon_0(\lambda, \theta,K_0)\in (0,1)$ such that for $q\in [2, 2+\varepsilon_0]$, $x_0\in \bR^2$, and $R\in (0,1]$, we have 
$$
\big(|D\bar{u}|^{q}+|\bar{p}|^q\big)^{1/q}_{B_{R/2}(x_0)}\le C\big(|D\bar{u}|^2+|\bar{p}|^2\big)^{1/2}_{B_{R}(x_0)}+C\big(|\bar{f}_\alpha|^{q}+|\bar{g}|^{q}\big)^{1/q}_{B_{R}(x_0)},
$$
where $C=C(\lambda,\theta,K_0,q)$.
Here, $\bar{u}$, $\bar{p}$, $\bar{f}_\alpha$, and $\bar{g}$ are the extensions of $u$, $p$, $f_\alpha$, and $g$ to $\bR^2$ so that they are zero on $\bR^2\setminus \Omega$.
\end{lemma}

\begin{proof}
By using Lemma \ref{170807@lem1} and Gehring's lemma, one can easily prove the lemma; see \cite[Lemma 3.8]{MR3758532}.
We omit the details here.
\end{proof}

In the lemma below, we prove the solvability of Stokes system in $\mathring{W}^1_q(\Omega)^2\times \tilde{L}_q(\Omega)$ when $q$ is close to $2$.

\begin{lemma}		\label{170518@lem1}
Let $\Omega$ satisfy \eqref{180403@eq2}.
Assume that Assumption \ref{A1} holds, and let 
$$
q\in \bigg[2-\frac{\varepsilon_0}{1+\varepsilon_0},2+\varepsilon_0\bigg],
$$
where  $\varepsilon_0=\varepsilon_0(\lambda, \theta,K_0)\in (0,1)$ is the constant from Lemma \ref{170807@lem2}.
Then for $f_\alpha\in L_q(\Omega)^2$ and $g\in \tilde{L}_q(\Omega)$, there exists a unique $(u,p)\in \mathring{W}^1_q(\Omega)^2\times \tilde{L}_q(\Omega)$ satisfying \eqref{170517_eq1}.
Moreover, we have
$$
\|Du\|_{L_q(\Omega)}+\|p\|_{L_q(\Omega)}\le C\big(\|f_\alpha\|_{L_q(\Omega)}+\|g\|_{L_q(\Omega)}\big),
$$
where $C=C(\lambda,\theta,K_0,\operatorname{diam}(\Omega),q)$.
\end{lemma}

\begin{proof}
Consider the following three cases:
$$
q=2, \quad q\in \bigg[2-\frac{\varepsilon_0}{1+\varepsilon_0}, 2\bigg), \quad q\in (2,2+\varepsilon_0].
$$
The first case follows from Lemma \ref{170517@lem1}.
The second case is a simple consequence of the last case combined with the duality argument; see the proof of \cite[Theorem 2.4]{MR3758532}.
Hence, here we only prove the case with $q\in (2,2+\varepsilon_0]$.

First, we assume that $f_\alpha\in L_\infty(\Omega)^2$ and $g\in \tilde{L}_\infty(\Omega)$.
By Lemma \ref{170517@lem1}, there exists a unique $(u,p)\in \mathring{W}^1_2(\Omega)^2\times \tilde{L}_2(\Omega)$ satisfying \eqref{170517_eq1} and \eqref{180403@eq1}.
Thus from Lemma \ref{170807@lem2}, we see that $(u, p)$ belongs to $\mathring{W}^1_q(\Omega)^2\times \tilde{L}_q(\Omega)$, and that 
\begin{align*}
&\|Du\|_{L_q(\Omega_{1/2}(x_0))}+\|p\|_{L_q(\Omega_{1/2}(x_0))}\\
&\le C\big(\|Du\|_{L_2(\Omega)}+\|p\|_{L_2(\Omega)}+\|f_{\alpha}\|_{L_q(\Omega)}+\|g\|_{L_q(\Omega)}\big)\\
&\le C\big(\|f_\alpha\|_{L_q(\Omega)}+\|g\|_{L_q(\Omega)}\big)
\end{align*}
for any $x_0\in \overline{\Omega}$, where $C=C(\lambda, \theta,K_0,\operatorname{diam}(\Omega), q)$.
By applying a covering argument, we obtain the desired  estimate.

To complete the proof, let $f_\alpha\in L_q(\Omega)^2$ and $g\in \tilde{L}_q(\Omega)$.
For $k\in \{1,2,\ldots\}$, we define $f_{\alpha,k}=(f_{\alpha,k}^1,\ldots,f_{\alpha,k}^d)^{\top}$ and $g_k$ by 
$$
f_{\alpha,k}^i=\max\{-k, \min\{f_\alpha^i,k\}\}, \quad g_k=\max\{-k,\min\{g,k\}\}.
$$
Since $f_{\alpha,k}$ and $g_k$ are bounded, by  the above result, there exists a unique $(u_k, p_k)\in \mathring{W}^1_q(\Omega)^2\times \tilde{L}_q(\Omega)$ satisfying 
\begin{equation}		\label{180403@A2}
\left\{
\begin{aligned}
\operatorname{div} u_k=g_k-(g_k)_\Omega &\quad \text{in }\, \Omega,\\
\cL u_k +\nabla p_k=D_\alpha f_{\alpha,k} &\quad \text{in }\, \Omega,
\end{aligned}
\right.
\end{equation}
and
\begin{equation}		\label{180403@A1}
\|Du_k\|_{L_q(\Omega)}+\|p_k\|_{L_q(\Omega)}\le C\big(\|f_{\alpha,k}\|_{L_q(\Omega)}+\|g_k\|_{L_q(\Omega)}\big),
\end{equation}
where $C=C(\lambda,\theta,K_0,\operatorname{diam}(\Omega),q)$.
Note that $f_{\alpha,k} \to f_\alpha$ and $g_k \to g$ in $L_q(\Omega)$ as $k\to \infty$.
Then  by \eqref{180403@A1},  $\{(u_k, p_k)\}$ is a Cauchy sequence in $\mathring{W}^1_q(\Omega)^2\times \tilde{L}_q(\Omega)$, and thus, there exists $(u, p)\in \mathring{W}^1_q(\Omega)^2\times \tilde{L}_q(\Omega)$ such that $(u_k, p_k)\to (u, p)$ in $\mathring{W}^1_q(\Omega)^2\times \tilde{L}_q(\Omega)$ and 
$$
\|Du\|_{L_q(\Omega)}+\|p\|_{L_q(\Omega)}\le C\big(\|f_{\alpha}\|_{L_q(\Omega)}+\|g\|_{L_q(\Omega)}\big).
$$
Finally, by taking the limit of the system \eqref{180403@A2}, it can easily be seen that $(u, p)$ satisfies \eqref{170517_eq1}.
Thus, the lemma is proved.
\end{proof}

\begin{remark}		\label{180427@rmk1}
One can extend the result in Lemma \ref{170518@lem1} to a system
\begin{equation}		\label{180427@eq1}
\left\{
\begin{aligned}
\operatorname{div} u=g &\quad \text{in }\, \Omega,\\
\cL u+\nabla p=f+D_\alpha f_\alpha &\quad \text{in }\, \Omega,
\end{aligned}
\right.
\end{equation}
when $q\in (2,2+\varepsilon_0]$ and $f\in L_{2q/(2+q)}(\Omega)^2$.
Indeed, if we fix $R>0$ such that $\Omega\subset B_R=B_R(0)$, then by \cite[Lemma 3.1]{MR3809039},  there exist ${F}_\alpha\in \tilde{W}^1_{2q/(2+q)}(B_R)^2$, $\alpha\in \{1,2\}$, satisfying 
$$
D_\alpha {F}_\alpha=f \chi_{\Omega} \quad \text{in }\, B_R,\\
$$
$$
\|F_\alpha\|_{L_{q}(B_R)}+\|DF_\alpha\|_{L_{2q/(2+q)}(B_R)}\le C(q)\|f\|_{L_{2q/(2+q)}(\Omega)},
$$
where $\chi_\Omega$ is the characteristic function.
Thus, by Lemma \ref{170518@lem1} applied to \eqref{180427@eq1} with $D_\alpha(F_\alpha + f_\alpha)$ in place of $f + D_\alpha f_\alpha$, there exists a unique $(u, p)\in \mathring{W}^1_q(\Omega)^2\times \tilde{L}_q(\Omega)$ satisfying \eqref{180427@eq1}.
Moreover, we have 
\begin{align}
\nonumber
\|Du\|_{L_{q}(\Omega)}+\|p\|_{L_q(\Omega)}&\le C\big(\|F_\alpha+f_\alpha\|_{L_q(\Omega)}+\|g\|_{L_q(\Omega)}\big)\\
\label{180427@eq3a}
&\le C\big(\|f\|_{L_{2q/(q+2)}(\Omega)}+\|f_\alpha\|_{L_q(\Omega)}+\|g\|_{L_q(\Omega)}\big),
\end{align}
where $C=C(\lambda, \theta,K_0, \operatorname{diam}(\Omega),q)$.
\end{remark}

We finish this section by establishing a weak $L_2$-estimate.

\begin{lemma}		\label{170516@lem1}
Let $\Omega$ satisfy \eqref{180403@eq2}.
Let $f_\alpha\in L_{2,\infty}(\Omega)^2$ and $g\in \tilde{L}_{2,\infty}(\Omega)$.
Then under Assumption \ref{A1},  there exists a unique $(u,p)$ belonging to
$$
\bigcap_{q\in [1,2)}\mathring{W}^1_q(\Omega)^2\times \tilde{L}_q(\Omega)
$$
and satisfying \eqref{170517_eq1}.
Moreover, $(Du, p)\in L_{2,\infty}(\Omega)^{2\times 2}\times \tilde{L}_{2,\infty}(\Omega)$ with the estimate 
\begin{equation}		\label{170516@eq2}
\|Du\|_{L_{2,\infty}(\Omega)}+\|p\|_{L_{2,\infty}(\Omega)}\le C\big(\|f_\alpha\|_{L_{2,\infty}(\Omega)}+\|g\|_{L_{2,\infty}(\Omega)}\big),
\end{equation}
where $C=C(\lambda,\theta, K_0,\operatorname{diam}(\Omega), q)$.
\end{lemma}

\begin{proof}
The proof of the lemma proceeds in a standard manner by applying the solvability result in Lemma \ref{170518@lem1} and an interpolation argument.
We present the proof for the sake of completeness.
Let $\varepsilon_0=\varepsilon_0(\lambda, \theta, K_0)\in (0,1)$ be the constant from Lemma \ref{170807@lem2}, and set 
$$
\varepsilon_1=\frac{\varepsilon_0}{1+\varepsilon_0}<\varepsilon_0.
$$
Then for each $q_0\in [2-\varepsilon_1, 2)$, by Lemma \ref{170518@lem1} and the fact that 
$$
f_\alpha\in L_{q_0}(\Omega)^2, \quad g\in \tilde{L}_{q_0}(\Omega),
$$
there exists a unique $(u,p)\in \mathring{W}^1_{q_0}(\Omega)^2\times \tilde{L}_{q_0}(\Omega)$ satisfying \eqref{170517_eq1}.
From the uniqueness of a  solution, the pair $(u, p)$ is unique and  belongs to $\mathring{W}^1_q(\Omega)^2\times \tilde{L}_q(\Omega)$ for all $q\in [1,2)$.
Set
$$
q_1=2-\varepsilon_1 \quad \text{and}\quad q_2=2+\varepsilon_1.
$$
For given $k>0$, we define $f_{\alpha}^+$ and $f_{\alpha}^-$ by 
$$
f_{\alpha}^+=
\left\{
\begin{aligned}
f_\alpha \quad \text{if }\, |f_\alpha |>k,\\
0 \quad \text{if }\, |f_\alpha|\le k,
\end{aligned}
\right.
\qquad 
f_{\alpha}^-=
\left\{
\begin{aligned}
0 \quad \text{if }\, |f_\alpha|>k,\\
f_{\alpha} \quad \text{if }\, |f_\alpha|\le k.
\end{aligned}
\right.
$$
Similarly, define $g^+\in L_{q_1}(\Omega)$ and $g^-\in L_{q_2}(\Omega)$.
Then we have
\begin{equation}		\label{170516@eq2a}
\begin{aligned}
f_\alpha^+\in L_{q_1}(\Omega)^2, \quad \|f_{\alpha}^+\|^{q_1}_{L_{q_1}(\Omega)}&\le \frac{2}{2-q_1}k^{q_1-2}\|f_\alpha\|_{L_{2,\infty}(\Omega)}^2, \\
f_\alpha^-\in L_{q_2}(\Omega)^2, \quad \|f_{\alpha}^-\|^{q_2}_{L_{q_2}(\Omega)}&\le \frac{q_2}{q_2-2}k^{q_2-2}\|f_\alpha\|_{L_{2,\infty}(\Omega)}^2,
\end{aligned}
\end{equation}
and
\begin{equation}		\label{170516@eq2b}
\begin{aligned}
g^+\in L_{q_1}(\Omega), \quad \|g_{1}\|^{q_1}_{L_{q_1}(\Omega)}&\le \frac{2}{2-q_1}k^{q_1-2}\|g\|_{L_{2,\infty}(\Omega)}^2,\\
g^-\in L_{q_2}(\Omega), \quad \|g_2\|^{q_2}_{L_{q_2}(\Omega)}&\le \frac{q_2}{2_2-q}k^{q_2-2}\|g\|_{L_{2,\infty}(\Omega)}^2.
\end{aligned}
\end{equation}
Indeed, for example, the first inequality of \eqref{170516@eq2a} follows from 
\begin{align*}
\|f_{\alpha}^+\|_{L_{q_1}(\Omega)}^{q_1}&=\bigg(\int_0^k +\int_k^\infty\bigg) q_1 t^{q_1-1}\big|\{x\in \Omega: |f_\alpha^+(x)|>t\}\big|\,dt\\
&\le q_1 k^{-2}\int_0^k t^{q_1-1}\,dt\cdot \|f_\alpha\|_{L_{2,\infty}(\Omega)}^2  + q_1\int_k^\infty t^{q_1-3}\,dt \cdot \|f_\alpha\|_{L_{2,\infty}(\Omega)}^{2}\\
&= \frac{2}{2-q_1} k^{q_1-2}\|f_\alpha\|_{L_{2,\infty}(\Omega)}^2.
\end{align*}
By Lemma \ref{170518@lem1}, there exists a unique
$(u^+,p^+)\in \mathring{W}^1_{q_1}(\Omega)^2\times \tilde{L}_{q_1}(\Omega)$ satisfying \eqref{170517_eq1} with $(f_{\alpha}^+, g^+)$ in place of $(f_{\alpha},g)$.
Moreover, we have 
$$
\|Du^+\|_{L_{q_1}(\Omega)}+\|p^+\|_{L_{q_1}(\Omega)}\le C\big(\|f_{\alpha}^+\|_{L_{q_1}(\Omega)}+\|g^+\|_{L_{q_1}(\Omega)}\big),
$$
where $C=C(\lambda,\theta,K_0,\operatorname{diam}(\Omega))$.
Using this together with \eqref{170516@eq2a} and \eqref{170516@eq2b}, we obtain for $t>0$ that 
\begin{align*}
t^2\big|\{x\in \Omega:|Du^+|+|p^+|>t/2\}\big|&\le 2^{q_1}t^{2-q_1} \int_\Omega \big(|Du^+|+|p^+|\big)^{q_1}\,dx \\
&\le Ct^{2-q_1}k^{q_1-2}\big(\|f_\alpha\|_{L_{2,\infty}(\Omega)}^2+\|g\|_{L_{2,\infty}(\Omega)}^2\big).
\end{align*}
Similarly, there exists a unique 
$(u^-,p^-)\in {W}^1_{q_2}(\Omega)^2\times \tilde{L}_{q_2}(\Omega)$
 satisfying \eqref{170517_eq1} with $(f_{\alpha}^-,g^-)$ in place of $(f_{\alpha},g)$, and  
$$
t^2\big|\{x\in \Omega:|Du^-|+|p^-|>t/2\}\big|\le C t^{2-q_2}k^{q_2-2}\big(\|f_\alpha\|_{L_{2,\infty}(\Omega)}^2+\|g\|_{L_{2,\infty}(\Omega)}^2\big).
$$
Combining these together and using the fact that  
$$
(u,p)=(u^+,p^+)+(u^-,p^-),
$$
we obtain
\begin{align*}
&t^2\big|\{x\in \Omega:|Du|+|p|>t\}\big|\\
&\le t^{2}\big(\big|\{x\in \Omega:|Du^+|+|p^+|>t/2\}\big|+\big|\{x\in \Omega:|Du^-|+|p^-|>t/2\}\big|\big)\\
&\le C(t^{2-q_1}k^{q_1-2}+t^{2-q_2}k^{q_2-2})\big(\|f_\alpha\|_{L_{2,\infty}(\Omega)}^2+\|g\|_{L_{2,\infty}(\Omega)}^2\big).
\end{align*}
Note that the above inequality holds for all $t,k>0$.
Therefore, by taking $k=t$, we get the desired estimate.
The lemma is proved.
\end{proof}

\section{Proofs of main results}				\label{S5}

\subsection{Proof of Theorem \ref{MT1}}		

Throughout this proof, we denote 
$$
q_0=2+\varepsilon_0 \quad \text{and} \quad q_1=\frac{2q_0}{2+q_0},
$$
where $\varepsilon_0=\varepsilon_0(\lambda, \theta, K_0)\in (0,1)$ is the constant from Lemma \ref{170807@lem2}.
Note that $q_0$ is the Sobolev conjugate of $q_1$, i.e. $q_0 = q_1^* = (2q_1)/(2-q_1)$.
We divide the proof into several steps.

{\em{Step 1.}}
In this step, we construct the Green function $(G, \Pi)$ satisfying the properties in Definition \ref{D1}.
We set
$$
\psi(x)=(\psi^1(x),\psi^2(x))^{\top}=\frac{x^{\top}}{2\pi |x|^2}.
$$
Then we have  
\begin{equation}		\label{170518_eq5}
\|\psi\|_{L_{2,\infty}(\bR^2)}\le \frac{1}{2\sqrt{\pi}}
\end{equation}
and 
$$
\operatorname{div}_x\psi(x-y)=\delta_y(x)
$$
in the sense that 
\begin{equation}		\label{171013@eq1}
\int_{\bR^2} \psi(x-y) \cdot \nabla\phi(x)\,dx=-\phi(y), \quad \forall y\in \bR^2, \quad \forall\phi\in C^\infty_0(\bR^2).
\end{equation}
For each $y\in \Omega$ and $\alpha,k\in \{1,2\}$, we set 
\begin{equation}		\label{170526_eq2}
f_{\alpha, y,k}(x)=-\psi^\alpha(x-y) e_k,
\end{equation}
where $e_k$ is the $k$-th unit vector in $\bR^2$.
By Lemma \ref{170516@lem1}, there exists $(v,\pi)=(v_{y,k}, \pi_{y,k})$ belonging to 
$$
\bigcap_{q\in[1,2)} \mathring{W}^1_q(\Omega)^2\times \tilde{L}_q(\Omega)
$$
and satisfying 
\begin{equation}		\label{170518_eq6}
\left\{
\begin{aligned}
\operatorname{div} v=0 &\quad \text{in }\, \Omega,\\
\cL v + \nabla \pi = D_\alpha f_{\alpha,y,k} &\quad \text{in }\, \Omega.
\end{aligned}
\right.
\end{equation}
Moreover, there is a version $\tilde{v}=\tilde{v}_{y,k}$ of $v$ such that $\tilde{v}=v$ a.e. in $\Omega$ and $\tilde{v}$ is continuous in $\Omega\setminus \{y\}$.
Indeed, 
by \eqref{171013@eq1} and \eqref{170518_eq6}, we see that  $(\eta v, \eta \pi)$ satisfies 
\begin{equation}		\label{180427@a1}
\left\{
\begin{aligned}
\operatorname{div} (\eta v)=\nabla \eta \cdot v &\quad \text{in }\, \Omega,\\
\cL (\eta v)+\nabla (\eta \pi)=F+D_\alpha F_\alpha &\quad \text{in }\, \Omega,
\end{aligned}
\right.
\end{equation}
where we set 
$$
F=A^{\alpha\beta}D_\beta v  D_\alpha \eta+\pi \nabla \eta, \quad 
F_\alpha=A^{\alpha\beta}D_\beta \eta v.
$$
Here, $\eta$ is a smooth function on $\bR^2$ satisfying 
$$
\eta\equiv 0\,\text{ on }\, B_{r/2}(y), \quad \eta\equiv 1 \, \text{ on }\, \bR^2\setminus B_r(y), \quad r>0.
$$
Since $F\in L_{q_1}(\Omega)^2$, $F_\alpha\in L_{q_0}(\Omega)^2$, and $\nabla \eta\cdot v\in L_{q_0}(\Omega)$,
by Remark \ref{180427@rmk1}, 
we have 
$$
(\eta v,\eta \pi)\in \mathring{W}^1_{q_0}(\Omega)^2\times L_{q_0}(\Omega),
$$
which implies that 
$$
(v,\pi)\in \mathring{W}^1_{q_0}(\Omega\setminus B_{r}(y))^2\times L_{q_0}(\Omega\setminus B_r(y)), \quad \forall r>0.
$$
Thus, by the Morrey-Sobolev embedding, 
there is a version $\tilde{v}=\tilde{v}_{y,k}$ of $v$ which is continuous in $\Omega \setminus \{y\}$.
We define a pair $(G,\Pi)$ by 
$$
G^{jk}(x,y)=\tilde{v}_{y,k}^j(x) \quad \text{and}\quad \Pi^k(x,y)=\pi_{y,k}(x).
$$
Here, $G$ is a $2\times 2$ matrix-valued function and $\Pi$ is a $1 \times 2$ vector-valued function on $\Omega\times \Omega$.

In the remainder of this step, we prove that $(G, \Pi)$ satisfies the properties $(i)$ -- $(iii)$ in Definition \ref{D1} so that $(G, \Pi)$ is the Green function (for the flow velocity) of $\cL$ in $\Omega$.
Clearly, the property $(i)$ holds.
To see the property $(ii)$, 
let $k\in \{1,2\}$ and $\phi\in \mathring{W}^1_\infty(\Omega)^2\cap C(\Omega)^2$.
Since $\operatorname{div} v_{y,k}=0$ in $\Omega$, the $k$-th column $G^{\cdot k}(\cdot,y)$ of $G(\cdot,y)$ satisfies $\operatorname{div} G^{\cdot k}(\cdot,y)=0$ in $\Omega$.
Notice from \eqref{171013@eq1} that 
$$
\begin{aligned}
\int_\Omega f_{\alpha,y,k} \cdot D_\alpha \phi\,dx&=-\int_\Omega \psi^\alpha(x-y)e_k \cdot D_\alpha \phi\,dx\\
&=-\int_{\Omega} \psi(x-y) \cdot \nabla \phi^k\,dx\\
&=\phi^k(y).
\end{aligned}
$$
From this equality and \eqref{170518_eq6}, it follows that 
$$
\int_\Omega A^{\alpha\beta}D_\beta G^{\cdot k}(\cdot ,y)\cdot D_\alpha \phi\,dx+\int_\Omega \Pi^k(\cdot ,y)\operatorname{div}\phi \,dx=\phi^k(y).
$$
To show the property $(iii)$, let $(u,p)\in \mathring{W}^1_2(\Omega)^2\times \tilde{L}_2(\Omega)$ be a weak solution of the adjoint problem \eqref{171010@eq1}.
By Lemma \ref{170518@lem1}, we see that 
\begin{equation}
							\label{eq0516_01}
(u,p)\in \mathring{W}^1_{q_1}(\Omega)^2\times \tilde{L}_{q_1}(\Omega) \quad \text{for some }\, q_1>2,
\end{equation}
and thus, by the Morrey-Sobolev embedding, there is a version of $u$, denoted by $\tilde{u}$, which is H\"older continuous in $\overline{\Omega}$.
From \eqref{eq0516_01} with $\tilde{u}$ in place of $u$ and  the fact that 
$$
(G(\cdot,y), \Pi(\cdot,y))\in \mathring{W}^1_q(\Omega)^2\times \tilde{L}_q(\Omega), \quad \forall q\in [1,2),
$$
we see that $\tilde{u}$ and $G^{\cdot k}(\cdot,y)$ are legitimate test functions to \eqref{171010@eq1} and \eqref{170518_eq6}, respectively.
By testing \eqref{170518_eq6} and \eqref{171010@eq1} with $\tilde{u}$ and $G^{\cdot k}(\cdot,y)$, respectively, we conclude that 
$$
\tilde{u}^k(y)=-\int_\Omega G^{\cdot k}(\cdot,y)\cdot f\,dx+\int_\Omega D_\alpha G^{\cdot k}(\cdot,y)\cdot f_\alpha\,dx+\int_\Omega \Pi^k(\cdot,y) g\,dx
$$
for all $y\in \Omega$.
Since $u=\tilde{u}$ a.e. in $\Omega$, the above identity implies  \eqref{170519@eq3}.
Thus, the property $(iii)$ holds.
Therefore, the pair $(G, \Pi)$ is the Green function of $\cL$ in $\Omega$.

{\em{Step 2.}}
In this step, we prove the assertions $(a)$ -- $(c)$ in Theorem \ref{MT1}.
The assertion $(a)$ follows immediately from \eqref{170516@eq2} and \eqref{170518_eq5}.

To prove the assertion $(b)$, 
we first claim that, for any $x\in \bR^2$, $y\in \Omega$, and $0<R<\operatorname{diam}(\Omega)$ satisfying $|x-y|>R$, we have
\begin{equation}		\label{180426@eq6a}
\begin{aligned}
&\|DG(\cdot,y)\|_{L_{q_0}(\Omega_{R/8}(x))} + \|\Pi(\cdot,y)\|_{L_{q_0}(\Omega_{R/8}(x))}\\
& \le CR^{-1}\big(\|DG(\cdot,y)\|_{L_{q_1}(\Omega_R(x))} +\|\Pi(\cdot,y)\|_{L_{q_1}(\Omega_R(x))}\big),
\end{aligned}
\end{equation}
where $C=C(\lambda, \theta, K_0, \operatorname{diam}(\Omega))$.
We consider the following two cases:
$$
B_{r}(x)\subset \Omega, \quad B_{r}(x)\cap \partial \Omega \neq \emptyset,
$$
where $r=R/4$.
\begin{enumerate}[i.]
\item
$B_{r}(x)\subset \Omega$. 
Let $\eta_1$ be a smooth function on $\bR^2$ satisfying 
$$
0\le \eta_1\le 1, \quad \eta_1\equiv 1 \, \text{ on }\, B_{r/2}(x), \quad \operatorname{supp}\eta_1 \subset B_{r}(x), \quad |\nabla \eta_1|\le 4r^{-1}.
$$
Then \eqref{180427@a1} holds with $v-(v)_{B_r(x)}$ and $\eta_1$ in place of $v$ and $\eta$.
Hence by \eqref{180427@eq3a} and the Poincar\'e inequality, we have 
$$
\begin{aligned}
&\|Dv\|_{L_{q_0}(B_{r/2}(x))} + \|\pi\|_{L_{q_0}(B_{r/2}(x))}\\
&\le C\big(\|F\|_{L_{q_1}(\Omega)}+\|F_\alpha\|_{L_{q_0}(\Omega)}+\|\nabla \eta_1\cdot (v-(v)_{B_r(x)})\|_{L_{q_0}(\Omega)}\big)\\
&\le Cr^{-1}\big(\|Dv\|_{L_{q_1}(B_r(x))}+\|\pi\|_{L_{q_1}(B_r(x))}\big),
\end{aligned}
$$
which gives \eqref{180426@eq6a}.
\item
$B_{r}(x)\cap \partial \Omega\neq \emptyset$.
We take $x_0\in B_r(x)\cap \partial \Omega$ such that $|x-x_0|=\operatorname{dist}(x, \partial \Omega)$, and observe that 
\begin{equation}		\label{180426@eq7}
B_{r/2}(x)\subset B_{3r/2}(x_0)\subset B_{3r}(x_0)\subset B_{4r}(x)=B_R(x).
\end{equation}
Let $\eta_2$ be a smooth function on $\bR^2$ satisfying 
$$
0\le \eta_2\le 1, \quad \eta_2\equiv 1 \, \text{ on }\, B_{3r/2}(x_0), \quad \operatorname{supp}\eta_2 \subset B_{3r}(x_0), \quad |\nabla \eta_2|\le 4r^{-1}.
$$
Since \eqref{180427@a1} holds with $\eta_2$ in place of $\eta$,
by using \eqref{180427@eq3a} and \eqref{180426@eq7}, we get 
$$
\begin{aligned}
&\|Dv\|_{L_{q_0}(\Omega_{r/2}(x))} + \|\pi\|_{L_{q_0}(\Omega_{r/2}(x))}\\
&\le C\big(\|F\|_{L_{q_1}(\Omega)}+\|F_\alpha\|_{L_{q_0}(\Omega)}+\|\nabla \eta_2\cdot v\|_{L_{q_0}(\Omega)}\big)\\
&\le Cr^{-1}\big(\|Dv\|_{L_{q_1}(\Omega_R(x))}+\|\pi\|_{L_{q_1}(\Omega_R(x))}\big),
\end{aligned}
$$
where we used the boundary Poincar\'e inequality together with \eqref{180405@eq4} in the last inequality.
This gives the inequality \eqref{180426@eq6a}.
\end{enumerate}

We are now ready to prove the estimates \eqref{180404@A2} and \eqref{180404@A3} in the assertion $(b)$.
Let $x\in \bR^2$, $y\in \Omega$, and $0<R<\operatorname{diam}(\Omega)$ with $|x-y|> R$.
We denote $M=\|DG(\cdot,y)\|_{L_{2,\infty}(\Omega_{R}(x))}$, and observe that 
$$
\begin{aligned}
&\|DG(\cdot,y)\|_{L_{q_1}(\Omega_{R}(x))}^{q_1}\\
&=q_{1}\bigg(\int_0^{M/R}+\int_{M/R}^\infty\bigg) t^{q_1-1}\big|\{z\in \Omega_R(x):|DG(\cdot,y)|>t\}\big|\,dt\\
&\le C R^{2-q_1} M^{q_1}.
\end{aligned}
$$
Similarly, we have 
$$
\|\Pi(\cdot,y)\|_{L_{q_1}(\Omega_{R}(x))}^{q_1}\le CR^{2-q_1}\|\Pi(\cdot,y)\|_{L_{2,\infty}(\Omega_R(x))}^{q_1}.
$$
By combining these together, and using \eqref{170807@eq6} and \eqref{180426@eq6a} with a covering argument, we obtain that
$$
\begin{aligned}
&\|DG(\cdot,y)\|_{L_{q_0}(\Omega_{R/2}(x))}+\|\Pi(\cdot,y)\|_{L_{q_0}(\Omega_{R/2}(x))}\\
&\le  CR^{-\gamma}\big(\|DG(\cdot,y)\|_{L_{2,\infty}(\Omega_{R}(x))}+\|\Pi(\cdot,y)\|_{L_{2,\infty}(\Omega_{R}(x))}\big)\\
&\le CR^{-\gamma},
\end{aligned}
$$
where $\gamma=2-2/q_1=1-2/q_0$ and $C=C(\lambda, \theta,K_0, \operatorname{diam}(\Omega))$.
This proves \eqref{180404@A2}.
We extend $G(\cdot,y)$ by zero on $\bR^2\setminus \Omega$.
Then using Morrey's inequality and the above inequality, we see that  
$$
[G(\cdot,y)]_{C^{\gamma}(B_{R/2}(x))} \le C\|DG(\cdot,y)\|_{L_{q_0}(B_{R/2}(x))}\le CR^{-\gamma}.
$$
This implies \eqref{180404@A3}, and thus the assertion $(b)$ is proved.
Note that by the above inequality, we have
\begin{equation}		\label{170524_eq1}
\left|G(z_0,y)-\dashint_{B_{R/2}(x)}G(z,y)\,dz\right|\le C_0
\end{equation}
for any $x,z_0\in \bR^2$, $y\in \Omega$, and $0<R<\operatorname{diam}(\Omega)$ satisfying $|x-y|>R$ and $z_0\in B_{R/2}(x)$, where $C_0=C_0(\lambda, \theta,K_0, \operatorname{diam}(\Omega))$.

We now turn to  the assertion $(c)$.
Let $x,y\in \Omega$ with $x\neq y$, and set 
$$
\rho=\frac{1}{4}|x-y|.
$$
Without loss of generality, we may assume that $x=0$ and $y=(-4\rho,0)$.
We choose a positive integer $k\ge 1$ satisfying
\begin{equation}		\label{170811@eq1}
2^k\rho < \frac{\operatorname{diam}(\Omega)}{2} \le 2^{k+1}\rho.
\end{equation}
For $i\in \{0, \ldots,k\}$, let $x_i=(\alpha_i,0)\in \bR^2$, where  
$$
\alpha_0=0, \quad \alpha_i=\alpha_{i-1}+2^{i}\rho = 2\rho(2^i-1), \quad i = 1, \ldots, k,
$$
and observe that 
$$
B_{2^i\rho}(x_i)\cap B_{2^{i-1}\rho}(x_{i-1})\neq\emptyset, \quad |x_i-y|>2^{i+1}\rho, \quad i\in \{1,\ldots,k\}.
$$
For each $i \in \{1, \ldots, k\}$, we choose $z_i \in B_{2^i \rho}(x_i) \cap B_{2^{i-1}\rho}(x_{i-1})$ and write
$$
\left|\dashint_{B_{2^{i-1}\rho}(x_{i-1})} G(z,y) \, dz\right| \leq \left|\dashint_{B_{2^{i-1}\rho}(x_{i-1})} G(z,y) \, dz - G(z_i,y)\right|
$$
$$
+ \left|G(z_i,y) - \dashint_{B_{2^i\rho}(x_i)} G(z,y) \, dz \right| + \left|\dashint_{B_{2^i\rho}(x_i)} G(z,y) \, dz\right|.
$$
Thanks to the estimate \eqref{170524_eq1}, this inequality implies that
$$
\left|\dashint_{B_{2^{i-1}\rho}(x_{i-1})}G(z,y)\,dz\right|\le 2C_0+\left|\dashint_{B_{2^{i}\rho}(x_{i})}G(z,y)\,dz\right|, \quad i\in \{1,\ldots,k\},
$$
and thus, by iterating we see that 
$$
\left|\dashint_{B_\rho(x)}G(z,y)\,dz\right|\le 2kC_0+\left|\dashint_{B_{2^k\rho}(x_k)}G(z,y)\,dz\right|.
$$
From this and \eqref{170524_eq1} we have
\begin{align}
\nonumber
|G(x,y)|&\le \left|G(x,y)-\dashint_{B_\rho(x)}G(z,y)\,dz\right|+\left|\dashint_{B_\rho(x)}G(z,y)\,dz\right|\\
\nonumber
&\le (2k+1)C_0+\left|\dashint_{B_{2^k\rho}(x_k)}G(z,y)\,dz\right|\\
\label{180405@eq1}
&\le C\bigg(\log \bigg(\frac{\operatorname{diam}(\Omega)}{\rho}\bigg)+\|G(\cdot,y)\|_{L_1(\Omega)}\bigg),
\end{align}
where the last inequality is due to the fact that (using \eqref{170811@eq1})
$$
1 \le k< \frac{1}{\log 2} \log \bigg(\frac{\operatorname{diam}(\Omega)}{\rho}\bigg), \quad |B_{2^k\rho}|\ge \bigg(\frac{\operatorname{diam}(\Omega)}{4}\bigg)^2.
$$
Notice from H\"older's inequality, the Sobolev inequality, and \eqref{170807@eq6} that 
\begin{align*}
\|G(\cdot,y)\|_{L_1(\Omega)}&\le C\|G(\cdot,y)\|_{L_2(\Omega)}\le C\|DG(\cdot,y)\|_{L_1(\Omega)}\\
&\le C\|DG(\cdot,y)\|_{L_{2,\infty}(\Omega)}\le C,
\end{align*}
where $C=C(\lambda, \theta,K_0,\operatorname{diam}(\Omega))$.
Therefore, from \eqref{180405@eq1} combined with the above inequalities we arrive at  \eqref{170807@eq7}.

{\em{Step 3.}}
In this step, we prove the identity \eqref{170526@eq1}.
For each $x\in \Omega$, we define the Green function $(G^*(\cdot ,x), \Pi^*(\cdot ,x))$ of the adjoint operator $\cL^*$ in the same manner that $(G(\cdot,y), \Pi(\cdot,y))$  is defined for the operator $\cL$. 
More precisely, we find a unique solution
$$
(w_{x,l}, \tau_{x,l})\in \bigcap_{q\in [1,2)} \mathring{W}^1_q(\Omega)^2\times \tilde{L}_q(\Omega)
$$
to the system  
$$
\left\{
\begin{aligned}
\operatorname{div} w_{x,l}=0 &\quad \text{in }\, \Omega,\\
\cL^* w_{x,l}+\nabla \tau_{x,l}=D_\alpha f_{\alpha,x,l} &\quad \text{in }\, \Omega,
\end{aligned}
\right.
$$
where $f_{\alpha,x,l}$ is the function as in \eqref{170526_eq2}.
Then we set $(w_{x,l}, \tau_{x,l})$ to be the $l$-th column of $(G^*(\cdot ,x), \Pi^*(\cdot ,x))$.
Using the arguments in {\em{Steps 1 and 2}}, we find that $(G^*, \Pi^*)$ satisfies the corresponding properties to those of $(G, \Pi)$.

Let $x,y\in \Omega$ with $x\neq y$, and denote $r=|x-y|/2$.
Let $\zeta$ be a smooth function in $\bR^2$ satisfying 
$$
\zeta\equiv 0\,\text{ on }\, B_{r/2}(x), \quad \zeta \equiv 1\, \text{ on }\, \bR^2 \setminus B_r(x).
$$
Observe that $\zeta G^*(\cdot,x)$ and $(1-\zeta)G^*(\cdot,x)$ can be applied to \eqref{170518_eq6} as test functions.
By testing the $l$-th columns of those functions to \eqref{170518_eq6}, and using the continuity of $G^*(\cdot,x)$ in $\Omega\setminus \{x\}$ and the fact that 
$$
G^*(\cdot,x)=\zeta G^*(\cdot,x)+(1-\zeta)G^*(\cdot,x), 
$$
we have 
$$
\int_\Omega A^{\alpha\beta}_{ij}D_\beta G^{jk}(\cdot,y)D_\alpha (G^*)^{il}(\cdot,x)\,dz=(G^*)^{kl}(y,x).
$$
Similarly, we obtain
$$
\int_\Omega A^{\alpha\beta}_{ij}D_\beta G^{jk}(\cdot,y)D_\alpha (G^*)^{il}(\cdot,x)\,dz=G^{lk}(x,y).
$$
By combining these together, we see that 
$$
G^{ l k}(x,y)=(G^*)^{k l}(y,x),
$$
which gives \eqref{170526@eq1}.
Finally, by the above identity and the continuity of $G^*(\cdot,x)$, it holds that $G(x,\cdot)$ is continuous in $\Omega\setminus \{x\}$.
By using this and the continuity of $G(\cdot,y)$ in $\Omega\setminus \{y\}$, we conclude that 
$G$ is continuous in $\{(x,y)\in \Omega\times \Omega:x\neq y\}$.
Thus, the theorem is proved.
\qed

\subsection{Proof of Theorem \ref{MT2}}		

To prove the theorem, we use the following interior $L^\infty$-estimate.

\begin{lemma}		\label{180405@lem1}
Let $R\in (0,1]$.
Suppose that the coefficients $A^{\alpha\beta}$ of $\cL$ are of partially Dini mean oscillation in $B_R=B_R(0)$ satisfying Definition \ref{D2} $(a)$ with a Dini function $\omega=\omega_A$.
If $(u, p)\in W^1_2(B_R)^2\times L_2(B_R)$ satisfies 
$$
\left\{
\begin{aligned}
\operatorname{div} u=0 &\quad \text{in }\, B_R,\\
\cL u+\nabla p=0 &\quad \text{in }\, B_R,
\end{aligned}
\right.
$$
then we have 
$$
(u, p)\in W^1_\infty(B_{R/2})^2\times L_\infty(B_{R/2})
$$
with the estimate
$$
\|Du\|_{L_\infty(B_{R/2})}+\|p\|_{L_\infty(B_{R/2})}\le CR^{-2}\big(\|Du\|_{L_1(B_R)}+\|p\|_{L_1(B_R)}\big),
$$
where $C=C(\lambda,\omega_A)$.
If we assume further that $A^{\alpha\beta}$ are of Dini mean oscillation with respect to all direction in $B_R$ satisfying Definition \ref{D2} $(b)$, then we have 
\begin{equation}		\label{180405@C1}
(u, p)\in C^1(\overline{B_{R/2}})^2\times C(\overline{B_{R/2}}).
\end{equation}
\end{lemma}

\begin{proof}
The lemma follows from \cite[Theorems 2.2 and 2.3, and Eq.\,(4.16)]{arXiv:1803.05560} together with scaling and covering arguments.
For more details, see \cite{CD3}.
\end{proof}

Let $x,y\in \Omega$ with $0<|x-y|\le \frac{1}{2}\operatorname{dist}(y, \partial \Omega)$.
Set $R=|x-y|/2$.
Since $y\notin B_R(x)$ and $B_R(x)\subset \Omega$, by the property $(ii)$ in Definition \ref{D1},   $(G^{\cdot k}(\cdot,y), \Pi^k(\cdot,y))$ satisfies 
$$
\left\{
\begin{aligned}
\operatorname{div} G^{\cdot k}(\cdot,y)=0 &\quad \text{in }\, B_R(x),\\
\cL G^{\cdot k}(\cdot,y)+\nabla \Pi^k (\cdot,y)=0 &\quad \text{in }\, B_R(x).
\end{aligned}
\right.
$$
By Lemma \ref{180405@lem1} and a covering argument (in case  $R>1$), 
we have 
\begin{equation}		\label{180405@A1}
\begin{aligned}
&\|DG(\cdot,y)\|_{L_\infty(B_{R/2}(x))}+\|\Pi(\cdot,y)\|_{L_\infty(B_{R/2}(x))}\\
&\le CR^{-2}\big(\|DG(\cdot,y)\|_{L_1(B_R(x))}+\|p\|_{L_1(B_R(x))}\big),
\end{aligned}
\end{equation}
where $C=C(\lambda,\operatorname{diam}(\Omega), \omega_A)$.
We denote $M=\|DG(\cdot,y)\|_{L_{2,\infty}(B_R(x))}$, and observe that  
\begin{align}
\nonumber
&\|DG(\cdot,y)\|_{L_1(B_R(x))}\\
\nonumber
&=\bigg(\int_0^{M/R}+\int_{M/R}^\infty\bigg) \big|\{z\in B_R(x):|D_xG(z,y)|>t\}\big|\,dt\\
\label{180405@D1}
&\le CRM.
\end{align}
Similarly, we have 
$$
\|\Pi(\cdot,y)\|_{L_1(B_R(x))}\le CR\|\Pi(\cdot,y)\|_{L_{2,\infty}(B_R(x))}.
$$
By combining these together, we get from \eqref{180405@A1} and \eqref{170807@eq6} that 
$$
\|DG(\cdot,y)\|_{L_\infty(B_{R/2}(x))}+\|\Pi(\cdot,y)\|_{L_\infty(B_{R/2}(x))}\le C R^{-1},
$$
where $C=C(\lambda, \theta,K_0,\operatorname{diam}(\Omega),\omega_A)$.
This gives the desired estimate.
Thus the theorem is proved.
\qed

\subsection{Proof of Theorem \ref{MT3}}		

To prove the theorem, we use the following estimate on a $C^{1,\rm{Dini}}$ domain.

\begin{lemma}		\label{180405@lem2}
Let $\Omega$ have a $C^{1,\rm{Dini}}$ boundary as in Definition \ref{D3}.
Suppose that the coefficients $A^{\alpha\beta}$ of $\cL$ are of Dini mean oscillation in $\Omega$ satisfying Definition \ref{D2} $(b)$ with a Dini function $\omega=\omega_A$.
Let $x_0\in \Omega$ and $0<R<\operatorname{diam}(\Omega)$.
If $(u, p)\in W^1_2(\Omega_R(x_0))^2\times L_2(\Omega_R(x_0))$ satisfies 
$$
\left\{
\begin{aligned}
\operatorname{div} u=0 &\quad \text{in }\, \Omega_R(x_0),\\
\cL u+\nabla p=0 &\quad \text{in }\, \Omega_R(x_0),\\
u=0 &\quad \text{on }\, \partial \Omega \cap B_R(x_0),
\end{aligned}
\right.
$$
then we have 
$$
(u, p)\in C^1(\overline{\Omega_{R/2}(x_0)})^2\times C(\overline{\Omega_{R/2}(x_0)})
$$
and 
$$
\begin{aligned}
&\|Du\|_{L_\infty(\Omega_{R/2}(x_0))}+\|p\|_{L_\infty(\Omega_{R/2}(x_0))}\\
&\le CR^{-3}\|u\|_{L_1(\Omega_{R}(x_0))}+CR^{-2}\big(\|Du\|_{L_1(\Omega_{R}(x_0))}+\|p\|_{L_1(\Omega_{R}(x_0))}\big),
\end{aligned}
$$
where $C=C(\lambda, \operatorname{diam}(\Omega), \omega_A, R_0, \varrho_0)$.
\end{lemma}

\begin{proof}
The lemma follows from \cite[Theorem 1.4 and Eq.\,(2.27)]{arXiv:1805.02506} with a localization argument.
For more details, see \cite{CD3}.
\end{proof}

Let $x,y\in \Omega$ with $x\neq y$ and $R=|x-y|/2$.
From the property $(ii)$ in Definition \ref{D1}, we see that
$$
\left\{
\begin{aligned}
\operatorname{div} G^{\cdot k}(\cdot,y)=0 &\quad \text{in }\, \Omega_R(x),\\
\cL G^{\cdot k}(\cdot,y)+\nabla \Pi^k (\cdot,y)=0 &\quad \text{in }\, \Omega_R(x),\\
G^{\cdot k}(\cdot,y)=0 &\quad \text{on }\, \partial \Omega \cap B_R(x).
\end{aligned}
\right.
$$
Then by Lemma \ref{180405@lem2}, we have 
$$
(G(\cdot,y), \Pi(\cdot,y))\in C^1(\overline{\Omega_{R/2}(x)})^{2\times 2}\times C(\overline{\Omega_{R/2}(x)})^2,
$$
which shows \eqref{eq0517_01}.
To prove the estimate \eqref{180405@F1}, we consider the following two cases:
$$
\partial \Omega\cap B_R(x) = \emptyset, \quad \partial \Omega\cap   B_R(x) \neq \emptyset.
$$
\begin{enumerate}[i.]
\item
$\partial \Omega \cap B_R(x)= \emptyset$.
Since $\Omega_R(x)=B_R(x)$, by following the proof of Theorem \ref{MT2}, we have 
$$
\|DG(\cdot,y)\|_{L_\infty(B_{R/2}(x))}+\|\Pi(\cdot,y)\|_{L_\infty(B_{R/2}(x))}\le C R^{-1},
$$
where $C=C(\lambda, \operatorname{diam}(\Omega),\omega_A, R_0, \varrho_0)$.
Together with the continuity of $DG(\cdot,y)$ and $\Pi(\cdot,y)$, this implies \eqref{180405@F1}.
\item
$\partial \Omega \cap B_R(x)\neq \emptyset$.
In this case, by Lemma \ref{180405@lem2}, we have 
\begin{equation}		\label{180405@E1}
\begin{aligned}
&\|DG(\cdot,y)\|_{L_\infty(\Omega_{R/2}(x))}+\|\Pi(\cdot,y)\|_{L_\infty(\Omega_{R/2}(x))}\le CR^{-3}\|G(\cdot,y)\|_{L_1(\Omega_{R}(x))}\\
&\quad +CR^{-2}\big(\|DG(\cdot,y)\|_{L_1(\Omega_{R}(x))}+\|\Pi(\cdot,y)\|_{L_1(\Omega_{R}(x))}\big),
\end{aligned}
\end{equation}
where $C=C(\lambda, \operatorname{diam}(\Omega), \omega_A, R_0, \varrho_0)$.
Note that (see \eqref{180405@D1})
\begin{equation}		\label{180406@eq1}
\begin{aligned}
&\|DG(\cdot,y)\|_{L_1(\Omega_{R}(x))}+\|\Pi(\cdot,y)\|_{L_1(\Omega_{R}(x))}\\
&\le CR\big( \|DG(\cdot,y)\|_{L_{2,\infty}(\Omega)}+\|\Pi(\cdot,y)\|_{L_{2,\infty}(\Omega)}\big)\le CR,
\end{aligned}
\end{equation}
where the last inequality is due to \eqref{170807@eq6}.
Fix a point $z_0\in \partial \Omega \cap B_R(x)$.
Since $G(z_0, y)=0$, we obtain by \eqref{180404@A3} that
$$
|G(z,y)|=|G(z,y)-G(z_0,y)|\le  CR^{-\gamma} |z-z_0|^\gamma\le C
$$
for all $z\in \Omega_R(x)$,
where $C=C(\lambda, \operatorname{diam}(\Omega), R_0, \varrho_0)$.
This implies 
$$
\|G(\cdot,y)\|_{L_1(\Omega_R(x))}\le CR^2,
$$
and thus, using \eqref{180405@E1} and \eqref{180406@eq1}, we conclude that 
$$
\|DG(\cdot,y)\|_{L_\infty(\Omega_{R/2}(x))}+\|\Pi(\cdot,y)\|_{L_\infty(\Omega_{R/2}(x))}\le CR^{-1}.
$$
Finally, by the continuity of $DG(\cdot,y)$ and $\Pi(\cdot,y)$, we get the desired estimate \eqref{180405@F1}.
\end{enumerate}
The theorem is proved.
\qed

\bibliographystyle{plain}

\end{document}